\documentclass{imanum}
\usepackage{graphicx}
\usepackage{color}
\jno{drnxxx}

\begin{document}

\title{Convergence of one-level and multilevel unsymmetric collocation for second order elliptic boundary value problems}
\shorttitle{}

\author{%
{\sc
Zhiyong Liu\thanks{ Email: zhiyong@nxu.edu.cn},
and Qiuyan Xu\thanks{Corresponding author. Email: qiuyanxu@nxu.edu.cn}
} \\[2pt]
School of Mathematics  and statistics,\\\
Ningxia University,\ Yinchuan 750021, Ningxia,\ China}
\shortauthorlist{Z. Liu and Q. Xu}

\maketitle

\begin{abstract}
{The paper proves convergence of one-level and multilevel unsymmetric collocation for second order elliptic boundary value problems on
the bounded domains. By using Schaback's linear discretization theory,
$L^{2}$ errors are obtained based on the kernel-based trial spaces generated by the compactly
supported radial basis functions. For the one-level unsymmetric collocation case,
we obtain convergence when the testing discretization is finer than the trial discretization.
The convergence rates depend on the regularity of the solution, the smoothness of the computing domain,
and the approximation of scaled kernel-based spaces.
The multilevel process is implemented by employing successive refinement scattered data sets and scaled compactly
supported radial basis functions with varying support radii. Convergence of multilevel
collocation is further proved based on the theoretical results of one-level unsymmetric collocation.
In addition to having the same dependencies as the one-level
collocation, the convergence rates of multilevel unsymmetric collocation especially depends on
the increasing rules of scattered data and the selection of scaling parameters.   }
{elliptic boundary value problems; radial basis functions; one-level collocation;
multilevel collocation; convergence rates.}
\end{abstract}

\section{Introduction}
\label{sec;introduction}
Since the first kernel-based collocation method was introduced by Kansa \citep[see][]{kansa86}, radial basis functions have been
successfully applied to solving various partial differential equations numerically.
But unfortunately, we are obtaining highly accurate
solutions from severely ill-conditioned linear systems and high computational cost
when using radial basis functions collocation methods with global supports.
To a certain extent, the utilization of compactly supported radial basis
functions (such as well-known Wendland's functions \citep[see][]{w95} and Wu's functions \citep[see][]{wu95})
improves the condition of collocation matrices. Then a multilevel interpolation algorithm
for the approximation of functions is developed, right after the popularity of
compactly supported radial basis functions \citep[see][]{floater}.
With a multilevel algorithm, the condition number of the discrete matrix can be relatively
small, and computation can be performed in $O(N)$ operations. In this algorithm,
the interpolation problem is solved
first on the coarsest level by one of the compactly supported radial basis functions
with a larger support (usually scaling the size of the support with the fill distance).
Then the residual can be formed, and be computed on the next finer level by
the same compactly supported radial basis function but with a smaller support.
This process can be repeated and stopped on the finest level. And the final
approximation is the sum of all of interpolants.
The sophisticated proofs of convergence of multilevel interpolation algorithm
were given for the first time on the sphere \citep[see][]{le10},
and on bounded domains \citep[see][]{w10}.
In fact, all one-level collocation methods can be used for constructing multilevel collocation algorithms,
such as multilevel symmetric collocation \citep[see][]{farell13,f99}, multilevel
Galerkin collocation \citep[see][]{ch14,w18},
and multilevel unsymmetric collocation (in this paper).
However, the situation becomes more complicated
when using multilevel algorithm to solve partial differential equations numerically.

One-level symmetric collocation method was developed by \citet{f97}, based on the generalized
Hermite interpolation method \citep[see][]{wu92}. The collocation matrices produced by
symmetric method are symmetric and positive definite when the trial data (centers)
and the testing data (collocation points) are the same.
Its convergence has been proved in \citet{franke} for generate
radial basis functions trial spaces and in \cite{farell13} for
some scaled kernel-based spaces. \citet{f99} suggests using multilevel symmetric collocation
to solve partial differential equations, and carries out a large number of
numerical observations. Convergence of the multilevel symmetric collocation for elliptic
boundary value problems was given in \citet{farell13}. Different from the multilevel interpolation algorithm,
the scaling parameter is no longer proportional to the fill distance in the multilevel symmetric collocation
method. This leads to a non-stationary multilevel approximation space.
In fact, the other two multilevel collocation methods also lead to the same phenomenon.
The multilevel symmetric collocation has been used for other scientific and engineering calculations
successfully \citep[see][]{chen,ch12,ch13,le12,farell17,townsend13,w17}.

One-level Galerkin collocation method was proposed by \citet{w98}. The method is constructed
by combining Ritz-Galerkin approximation with radial basis functions. \citet{w98} proves that
the Galerkin collocation method leads to the same error bounds as the classical finite element methods
under $W^{1,2}(\Omega)$ norm. By using the Aubin-Nische technology, \citet{w18} gives $L^{2}(\Omega)$ error
of one-level Galerkin collocation method. The idea of combing a multilevel interpolation
algorithm with the Galerkin approximation was suggested in \citet{w98} first. Then convergence
of multilevel Galerkin collocation
method was given in \citet{ch14} and further refined in \cite{w18}.

One-level unsymmetric collocation method was proposed by \citet{kansa86}, which
has a little earlier history than symmetric collocation and
Galerkin collocation methods. Compared with the other two collocation methods,
one-level unsymmetric collocation method is more flexible and concise. It
avoids constructing the trial space based on complicated differential operator
or implementing the numerical integration. However, \citet{hon} shows that
this method may be failure for some
scattered data sites with the special distribution. \cite{s07,s10}
prove that one-level unsymmetric collocation method is convergent once
the testing discretization has been selected to have more degrees of freedom than the trial
discretization. By using Schaback's linear discretization theory, we will
prove the $L^{2}$ error of one-level unsymmetric collocation method. Further, we
will construct multilevel unsymmetric collocation method and prove
its convergence in the paper.

The paper is organized as follows. Section 2 is devoted to some notation and
the second order elliptic partial differential equation model. Section 3 is
an introduction of kernel-based trial spaces generated by compactly supported
radial basis functions. A Bernstein inequality for functions from scaled kernel-based
trial spaces will be given in this section. Section 4 and Section 5 are convergence
of one-level unsymmetric collocation method and multilevel unsymmetric collocation method.
Section 6 closes the paper with a concise conclusion.

\section{ Preliminaries}
We assume $\Omega\subseteq\mathbb{R}^{d}$ is a bounded domain. We will provide the $L^{2}(\Omega)$ convergence analysis of one-level and multilevel unsymmetric collocation for elliptic boundary value problems
in the paper. So we start by introducing some notation about integer order and fractional order Sobolev spaces.
\subsection{Sobolev norms}
With a nonnegative integer $k$ and $1\leq p<\infty$, Sobolev space $W^{k,p}(\Omega)$ contains all functions which have weak differentiability order
$k$ and integrability power $p$. $W^{k,p}(\Omega)$ has been assembled with the (semi-) norms
$$|u|_{W^{k,p}(\Omega)}=(\sum_{|\alpha|=k}\|D^{\alpha}u\|_{L^{p}(\Omega)}^{p})^{1/p},\quad
\|u\|_{W^{k,p}(\Omega)}=(\sum_{|\alpha|\leq k}\|D^{\alpha}u\|_{L^{p}(\Omega)}^{p})^{1/p}.$$
When $p=\infty$, the (semi-) norms are defined as
$$|u|_{W^{k,\infty}(\Omega)}=\sup_{|\alpha|=k}\|D^{\alpha}u\|_{L^{\infty}(\Omega)},\quad
\|u\|_{W^{k,\infty}(\Omega)}=\sup_{|\alpha|\leq k}\|D^{\alpha}u\|_{L^{\infty}(\Omega)}.$$
When $p=2$ and noninteger values $0<\sigma<\infty$, the fractional order Sobolev space $W^{\sigma,2}(\mathbb{R}^{d})$ is characterized by Fourier transform
\begin{equation}\label{ft}
\widehat{u}(\boldsymbol{\omega})=(2\pi)^{-d/2}\int_{\mathbb{R}^{d}}u(\textbf{x})e^{-i\textbf{x}^{T}\boldsymbol{\omega}}d\textbf{x}.
\end{equation}
At this time the norm of $W^{\sigma,2}(\mathbb{R}^{d})$ is given by
\begin{equation}\label{snorm}
\|u\|_{W^{\sigma,2}(\mathbb{R}^{d})}^{2}=\int_{\mathbb{R}^{d}}
|\widehat{u}(\boldsymbol{\omega})|^{2}(1+\|\boldsymbol{\omega}\|_{2}^{2})^{-\sigma}d\boldsymbol{\omega}.
\end{equation}
$W^{\sigma,2}(\Omega)$ is a subspace of $W^{\sigma,2}(\mathbb{R}^{d})$, which can be simply defined by
$$W^{\sigma,2}(\Omega)=\{u|_{\Omega}: u\in W^{\sigma,2}(\mathbb{R}^{d})\}.$$
We always assume $\Omega$ has a $C^{k,s}-$boundary with $s\in[0,1)$, thus we can define the Sobolev norm on $\partial\Omega$ by
a number of $C^{k,s}-$diffeomorphisms
$$\psi_{j}:B\rightarrow V_{j}, \quad 1\leq j \leq K,$$
where $B=B(\textbf{0},1)$ is the unit ball in $\mathbb{R}^{d-1}$ and $V_{j}\subseteq \mathbb{R}^{d}$ are
open sets such that $\partial\Omega=\bigcup_{j=1}^{K}V_{j}$. Hence the Sobolev norm on the boundary can be
defined by
$$\|u\|_{W^{\sigma,p}(\partial\Omega)}^{2}=\sum_{j=1}^{K}\|(u\omega_{j})\circ \psi_{j}\|_{W^{\sigma,p}(B)}^{p},$$
where $\{\omega_{j}\}$ is a partition of unity with respect to $V_{j}$.

The connection between $W^{\sigma,2}(\Omega)$ and $W^{\sigma,2}(\mathbb{R}^{d})$ can be
characterized by Sobolev extension \citep[see][]{adams}.
\begin{lemma}\label{skuo}(\textbf{Sobolev extension})
Let $\Omega\subseteq \mathbb{R}^{d}$ be open and have a Lipschitz boundary. Let $\sigma\geq0$.
Then there exists a linear operator $E: W^{\sigma,2}(\Omega)\rightarrow W^{\sigma,2}(\mathbb{R}^{d})$
such that for all $f\in W^{\sigma,2}(\Omega)$ the following results hold:

(1) $Ef|_{\Omega}=f$,

(2)$\|Ef\|_{W^{\sigma,2}(\mathbb{R}^{d})}\leq C\|f\|_{W^{\sigma,2}(\Omega)}$.
\label{l21}
\end{lemma}
To establish a relation between Sobolev norm for the domain $\Omega$ and Sobolev norm for the boundary
$\partial\Omega$, we need to cite the trace theorem and inverse trace theorem.
\begin{lemma}\label{trace}(\textbf{trace theorem})Suppose $\Omega\subseteq\mathbb{R}^{d}$ is a bounded domain with
$C^{k,s}-$boundary. Let $1/2<\sigma<k+s$. Then, there exists a continuous linear operator
$$T:W^{\sigma,2}(\Omega)\rightarrow W^{\sigma-1/2,2}(\partial\Omega)$$
such that $Tu=u|_{\partial\Omega}$ for all $u\in W^{\sigma,2}(\Omega)$ and
$$\|u\|_{W^{\sigma-1/2,2}(\partial\Omega)}\leq C\|u\|_{W^{\sigma,2}(\partial\Omega)}$$
for some positive constant $C$ which is independent of $u$.
\end{lemma}
The proof of trace theorem can be found in \citet{wl87}. The following inverse trace theorem
depicts the opposite situation, namely, the linear mapping from
$W^{\sigma-1/2,2}(\partial\Omega)$ to $W^{\sigma,2}(\Omega)$.
\begin{lemma}
\label{l23}
(\textbf{inverse trace theorem})Suppose $\Omega\subseteq\mathbb{R}^{d}$ is a bounded region with
$C^{k,s}-$boundary. Let $1/2<\sigma<k+s$. Then, there exists a continuous linear operator
$$Z:W^{\sigma-1/2,2}(\partial\Omega)\rightarrow W^{\sigma,2}(\Omega)$$
such that $T\circ Zu=u$ for all $u\in W^{\sigma-1/2,2}(\partial\Omega)$.
\end{lemma}
\subsection{Elliptic boundary value problems and regularity}
\label{sec;Sobolev}
We consider a second order elliptic boundary value problem with the form
\begin{align}
Lu & = f
\quad \hbox{in}\ \Omega,
\label{31} \\[4pt]
u & = g
\quad \hbox{on}\ \partial\Omega.
\label{32}
\end{align}
where $L$ is strictly elliptic on $\Omega$ with the form
\begin{equation}
Lu(\textbf{x})=\sum_{i,j=1}^{d}a_{ij}(\textbf{x})\frac{\partial^{2}}{\partial x_{i}\partial x_{j}}u(\textbf{x})+
\sum_{i=1}^{d}b_{i}(\textbf{x})\frac{\partial}{\partial x_{i}}u(\textbf{x})+c(\textbf{x})u(\textbf{x}).
\label{33}
\end{equation}

We assume that the solution of (\ref{31})-(\ref{32}) has the regularity $u\in W^{\sigma,2}(\Omega)$
with $\sigma> d/2+2$. When making appropriate assumptions about the coefficients of (\ref{33}),
we can obtain Lemma \ref{llianxu} which describes the continuity of linear operator $L$. The proof can be
found in \cite{giesl}.
\begin{lemma}
\label{llianxu}
(\textbf{continuity of $L$})Let $\sigma\in\mathbb{R}$, and $k=\lfloor\sigma\rfloor>d/2+2$.
Let the coefficients of (\ref{33}) satisfy $a_{ij},b_{i},c\in W^{k-1,\infty}(\Omega)$. Then $L$ is a bounded
operator from $W^{\sigma,2}(\Omega)$ to $W^{\sigma-2,2}(\Omega)$ which satisfies
$$\|Lu\|_{W^{\sigma-2,2}(\Omega)}\leq C\|u\|_{W^{\sigma,2}(\Omega)},\quad  u\in W^{\sigma,2}(\Omega).$$
\end{lemma}
Another important theoretical result about the equations (\ref{31})-(\ref{32}) is priori inequality, which
describes the continuously dependent properties of the solution on the data.
\begin{lemma}\label{xianyan}(\textbf{priori inequality})
Let $u\in W^{1,2}(\Omega)$ satisfy the equations (\ref{31})-(\ref{32}) in the weak sense. Then
$g\in W^{1/2,2}(\partial\Omega)$, and there exists a constant $C$ such that
\begin{equation}
\label{24}
\|u\|_{L^{2}(\Omega)}\leq C\{\|f\|_{L^{2}(\Omega)}+\|g\|_{W^{1/2,2}(\partial\Omega)}\}.
\end{equation}
\end{lemma}

\begin{proof}
With the help of the inverse trace operator $Z$, we know
$Zg\in W^{1,2}(\Omega)$. Using the $L^{2}-$analogue of the maximum principle taken from
\citet{gilbarg}, then there exists a constant $C_{1}$ such that
\begin{equation}
\|u\|_{W^{1,2}(\Omega)}\leq C_{1}\{\|f\|_{L^{2}(\Omega)}+\|Zg\|_{W^{1,2}(\Omega)}\}.
\end{equation}
From the continuity of the inverse trace operator (Lemma \ref{l23}), we have
$$\|Zg\|_{W^{1,2}(\Omega)}\leq C_{2}\|g\|_{W^{1/2,2}(\partial\Omega)}.$$
Using the trivial estimate $\|u\|_{L^{2}(\Omega)}\leq\|u\|_{W^{1,2}(\Omega)}$ yields the inequality (\ref{24}).
\end{proof}

\section{Kernel-based trial spaces}
We want to use the trial spaces produced by radial basis functions as finite dimensional
approximation of the equations (\ref{31})-(\ref{32}). In fact, such trial spaces are the subspaces
of the reproducing kernel Hilbert spaces (or native spaces). The reproducing kernel plays an important
role in our later analysis.
\begin{definition}
Let $H$ be a real Hilbert space of functions $f:\Omega(\subseteq\mathbb{R}^{d})\rightarrow\mathbb{R}$
with inner product $\langle\cdot,\cdot\rangle_{H}$. A function $\Phi:\Omega\times\Omega\rightarrow\mathbb{R}$
is called reproducing kernel for $H$ if

(1)$\Phi(\cdot,\textbf{x})\in H$ for all $\textbf{x}\in\Omega$,

(2)$f(\textbf{x})=\langle f,\Phi(\cdot,\textbf{x})\rangle_{H}$ for all $f\in H$ and all $\textbf{x}\in\Omega$.
\end{definition}
A translation-invariant kernel such as the strictly positive definite radial function $\Phi(\textbf{x},\textbf{y})=\Phi(\textbf{x}-\textbf{y})$
will generate a reproducing kernel Hilbert space, the so-called native space $\mathcal{N}_{\Phi}(\Omega)$.
But we will make use of the native space $\mathcal{N}_{\Phi}(\mathbb{R}^{d})$ which defined in $\mathbb{R}^{d}$ and
described by Fourier transform (\ref{ft}). $\mathcal{N}_{\Phi}(\mathbb{R}^{d})$ consists of all functions $f\in L^{2}(\mathbb{R}^{d})$
with
\begin{equation}\label{nnorm}
\|f\|_{\Phi}^{2}=\int_{\mathbb{R}^{d}}\frac{|\widehat{f}(\boldsymbol{\omega})|^{2}}
{\widehat{\Phi}(\boldsymbol{\omega})}d\boldsymbol{\omega}<\infty.
\end{equation}
Hence, if the Fourier transform of $\Phi:\mathbb{R}^{d}\rightarrow\mathbb{R}$ has the algebraic decay condition
\begin{equation}\label{adc}
c_{1}(1+\|\boldsymbol{\omega}\|_{2}^{2})^{-\sigma}\leq \widehat{\Phi}\leq c_{2}(1+\|\boldsymbol{\omega}\|_{2}^{2})^{-\sigma}
\end{equation}
for two fixed constants $0<c_{1}<c_{2}$, then the native space $\mathcal{N}_{\Phi}(\mathbb{R}^{d})$ will be norm equivalent to
the Sobolev space $W^{\sigma,2}(\mathbb{R}^{d})$. In the paper, we always use $\Phi$ to represent the radial functions which satisfy
the condition (\ref{adc}). Thus, with a radial function $\Phi$ satisfied (\ref{adc}) and a scaling parameter
$\delta\in(0,1]$ we can construct a scaled radial function
\begin{equation}\label{rbf}
\Phi_{\delta}:=\delta^{-d}\Phi((\textbf{x}-\textbf{y})/\delta), \quad \textbf{x},\textbf{y}\in\mathbb{R}^{d}.
\end{equation}
The following Lemma establishes a relation between the native space $\mathcal{N}_{\Phi_{\delta}}(\mathbb{R}^{d})$ and the
Sobolev space $W^{\sigma,2}(\mathbb{R}^{d})$. Its proof can be found in \citet{w10}.
\begin{lemma}\label{mdj}(\textbf{norm equivalence})
For every $\delta\in(0,1]$ we have $\mathcal{N}_{\Phi_{\delta}}(\mathbb{R}^{d})=W^{\sigma,2}(\mathbb{R}^{d})$.
And we have for every $f\in W^{\sigma,2}(\mathbb{R}^{d})$ the norm equivalence
\begin{equation}
c_{1}\|f\|_{\Phi_{\delta}}\leq\|f\|_{W^{\sigma,2}(\mathbb{R}^{d})}\leq c_{2}\delta^{-\sigma}\|f\|_{\Phi_{\delta}}
\end{equation}
with $c_{1},c_{2}>0$.
\end{lemma}
Now we can consider the finite dimensional kernel-based trial space generated by $\Phi_{\delta}$. Let
$Y^{I}=\{\textbf{y}_{1},\textbf{y}_{2},\ldots,\textbf{y}_{N_{I}}\}\subseteq \Omega$ and $Y^{B}=\{\textbf{y}_{N_{I}+1},\textbf{y}_{N_{I}+2},\ldots,\textbf{y}_{N}\}\subseteq \partial\Omega$
be some data sites which are arranged on the domain and its boundary respectively. Hence the whole of centers are
$Y=Y^{I}\bigcup Y^{B}$, which can used to construct a trial space and determine its degrees of freedom.
Some measures are defined as
follows:
$$h_{I}:=\sup_{\textbf{y}\in\Omega}\min_{\textbf{y}_{j}\in Y^{I}}\|\textbf{y}-\textbf{y}_{j}\|_{2},
\quad \textrm{for interior centers} \ Y^{I},$$
$$h_{B}:=\max_{1\leq j\leq K}h_{Bj}=\max_{1\leq j\leq K}\left\{\sup_{\textbf{y}\in B=B(\textbf{0},1)}
\min_{\textbf{y}_{j}\in \psi_{j}^{-1}\left(Y^{B}\bigcap V_{j}\right)}\|\textbf{y}-\textbf{y}_{j}\|_{2}\right\},
\quad \textrm{for boundary centers} \ Y^{B},$$
$$q_{I}:=\frac{1}{2}\min_{\textbf{y}_{i},\textbf{y}_{k}\in Y^{I},i\neq k}\|\textbf{y}_{i}-\textbf{y}_{k}\|_{2},
\quad \textrm{for interior centers}\ Y^{I},$$
$$q_{B}:=\min_{1\leq j\leq K}q_{Bj}=\min_{1\leq j\leq K}\left\{\frac{1}{2}\min_{\textbf{y}_{i},\textbf{y}_{k}
\in \psi_{j}^{-1}\left(Y^{B}\bigcap V_{j}\right),i\neq k}\|\textbf{y}_{i}-\textbf{y}_{k}\|_{2}\right\},
\quad \textrm{for boundary centers} \ Y^{B},$$
$$h_{Y}:=\max\{h_{I},h_{B}\},\quad q_{Y}:=\min\{q_{I},q_{B}\}.$$
Using the scaled radial function $\Phi_{\delta}$ as basis functions, then the kernel-based trial space is
\begin{equation}
\label{ts}
W_{\delta}=\textrm{span}\{\Phi_{\delta}(\cdot-\textbf{y})|\textbf{y}\in Y\}.
\end{equation}

To establish the Bernstein-type inequality in the kernel-based trial space $W_{\delta}$, we need a class of
band-limited functions. Let $\rho>0$, we define $\mathcal{B}_{\rho}$ to be
$$\mathcal{B}_{\rho}:=\{f\in L^{2}(\mathbb{R}^{d}): \textrm{supp}\ \widehat{f}\subseteq B(\textbf{0},\rho)\},$$
where $B(\textbf{0},\rho)$ is the ball in $\mathbb{R}^{d}$ having center $\textbf{0}$ and radius $\rho$.
We denote the reproducing kernel of $W^{\beta,2}(\mathbb{R}^{d})$ by $\Psi:\mathbb{R}^{d}\rightarrow\mathbb{R}$
which has a Fourier transform that satisfies
\begin{equation}\label{psi}
\widehat{\Psi}=(1+\|\boldsymbol{\omega}\|_{2}^{2})^{-\beta},\quad \boldsymbol{\omega}\in\mathbb{R}^{d},
\end{equation}
where we take $\sigma\geq\beta>d/2$. At this time, the scaled kernel is
\begin{equation}\label{rbf}
\Psi_{\delta}:=\delta^{-d}\Psi((\textbf{x}-\textbf{y})/\delta), \quad \textbf{x},\textbf{y}\in\mathbb{R}^{d}.
\end{equation}
We first need an auxiliary result from \citet{nar06}.
\begin{lemma}\label{nar33}
Let $f=\sum_{j=1}^{N}c_{j}\Psi(\cdot-\textbf{x}_{j}), \ \textbf{x}_{j}\in Y$. Define $\widehat{f_{\rho}}=\widehat{f}_{YB(0,\rho)}$,
where $YB(0,\rho)$ is the characteristic function of the ball $B(0,\rho)$. Then, there exists a constant $k>0$,
which is dindependent of $Y$ and the $\{c_{j}\}$, such that for $\rho=k/q_{Y}$ the following inequality holds
\begin{equation}\label{rbf}
\|f-f_{\rho}\|_{W^{\beta,2}(\mathbb{R}^{d})}\leq\frac{1}{2}\|f\|_{W^{\beta,2}(\mathbb{R}^{d})}.
\end{equation}
\end{lemma}
This Lemma can be used to prove a version of scaled trial spaces generated by $\Psi_{\delta}$.
\begin{lemma}\label{nar33n}
Let $f=\sum_{j=1}^{N}c_{j}\Psi_{\delta}(\cdot-\textbf{x}_{j}), \ \textbf{x}_{j}\in Y$. Define
$\widehat{f_{\rho/\delta}}=\widehat{f}_{YB(0,\rho/\delta)}$,
where $YB(0,\rho/\delta)$ is the characteristic function of the ball $B(0,\rho/\delta)$. Then there exists a constant $k>0$,
which is dindependent of $Y$ and the $\{c_{j}\}$, such that for $\rho=k\delta/q_{Y}$ the following inequality holds
\begin{equation}\label{nar11}
\|f-f_{\rho/\delta}\|_{\Psi_{\delta}}\leq\frac{1}{2}\|f\|_{\Psi_{\delta}},
\end{equation}
and
\begin{equation}\label{nar12}
\|f\|_{\Psi_{\delta}}\leq 2\|f_{\rho/\delta}\|_{\Psi_{\delta}}.
\end{equation}
\end{lemma}
\begin{proof}
The proof is very similar to the technical process of Lemma 4 of \citet{w10}.
Define $g(\textbf{x})=\delta^{d/2}f(\delta\textbf{x})$, then $g\in W^{\beta,2}(\mathbb{R}^{d})$
and $\|g\|_{W^{\beta,2}(\mathbb{R}^{d})}=\|f\|_{\Psi_{\delta}}$.
Define a new scattered data set $X=Y/\delta=\{\textbf{y}/\delta: \textbf{y}\in Y\}$
which has separation distance $q_{X}=q_{Y}/\delta$.  Then we can define a $g_{\rho}$
with $\rho=k\delta/q_{Y}$ and $\widehat{g_{\rho}}=\widehat{g}_{XB(\textbf{0},\rho)}$. By using
Lemma \ref{nar33} we have
\begin{equation*}
\|g-g_{\rho}\|_{W^{\beta,2}(\mathbb{R}^{d})}\leq\frac{1}{2}\|g\|_{W^{\beta,2}(\mathbb{R}^{d})}.
\end{equation*}
We then define $f_{\rho/\delta}(\textbf{x})=\delta^{-d/2}g_{\rho}(\textbf{x}/\delta)$, and for
$\boldsymbol{\omega}\in B(0,\rho/\delta)\subset\mathbb{R}^{d}$
$$\widehat{f_{\rho/\delta}}(\boldsymbol{\omega})=\delta^{d/2}\widehat{g_{\rho}}(\delta\boldsymbol{\omega})
=\delta^{d/2}\widehat{g}_{XB(\textbf{0},\rho)}(\delta\boldsymbol{\omega})=\widehat{f}_{YB(0,\rho/\delta)}(\boldsymbol{\omega}).$$
Since $\|f_{\rho/\delta}\|_{\Psi_{\delta}}=\|g_{\rho}\|_{W^{\beta,2}(\mathbb{R}^{d})}$, so
\begin{align*}
\|f-f_{\rho/\delta}\|_{\Psi_{\delta}}&=\|g-g_{\rho}\|_{W^{\beta,2}(\mathbb{R}^{d})}
 \\[4pt]
 & \leq\frac{1}{2}\|g\|_{W^{\beta,2}(\mathbb{R}^{d})}\leq\frac{1}{2}\|f\|_{\Psi_{\delta}}.
\end{align*}
(\ref{nar12}) is a direct result from (\ref{nar11}).
\end{proof}

We can now combine Lemma \ref{nar33n} with Lemma \ref{mdj} to arrive at the following inverse inequality.
\begin{lemma}\label{invv}(\textbf{inverse inequality})
Let $\sigma\geq\beta>d/2$. Let $f=\sum_{j=1}^{N}c_{j}\Psi_{\delta}(\cdot-\textbf{x}_{j}), \ \textbf{x}_{j}\in Y$.
Let $\rho=k\delta/q_{Y}$ with $k>0$. Then
there exists a constant $C$ such that
\begin{equation}\label{inv}
\|f\|_{W^{\beta,2}(\mathbb{R}^{d})}\leq Cq_{Y}^{-\beta}\|f\|_{L^{2}(\mathbb{R}^{d})}.
\end{equation}
\end{lemma}
\begin{proof}
Using Lemma \ref{nar33n}, Lemma \ref{mdj}, (\ref{psi}) and Parseval's identity,
\begin{align*}
\|f\|_{W^{\beta,2}(\mathbb{R}^{d})}^{2}&\leq c\delta^{-2\beta}\|f\|_{\Psi_{\delta}}^{2}
 \\[4pt]
 & \leq4c\delta^{-2\beta}\int_{\|\boldsymbol{\omega}\|_{2}\leq \rho/\delta}
 \frac{|\widehat{f}(\boldsymbol{\omega})|^{2}}
{\widehat{\Psi_{\delta}}(\boldsymbol{\omega})}d\boldsymbol{\omega}
\\[4pt]
&=4c\delta^{-2\beta}\int_{\|\boldsymbol{\omega}\|_{2}\leq \rho/\delta}
|\widehat{f}(\boldsymbol{\omega})|^{2}(1+\delta^{2}\|\boldsymbol{\omega}\|_{2}^{2})^{\beta}d\boldsymbol{\omega}
\\[4pt]
&\leq4c\delta^{-2\beta}(1+\rho^{2})^{\beta}\int_{\|\boldsymbol{\omega}\|_{2}\leq \rho/\delta}
|\widehat{f}(\boldsymbol{\omega})|^{2}d\boldsymbol{\omega}
\\[4pt]
&\leq4c\delta^{-2\beta}2^{\beta}\rho^{2\beta}\int_{\mathbb{R}^{d}}
|\widehat{f}(\boldsymbol{\omega})|^{2}d\boldsymbol{\omega}
\\[4pt]
&= 4c2^{\beta}k^{2\beta}q_{Y}^{-2\beta}\|f\|_{L^{2}(\mathbb{R}^{d})}^{2}.
\end{align*}
\end{proof}

We also need another auxiliary result comes from \cite{w18}, which is a scaled version
of Theorem 3.4 in \citet{nar06}.
\begin{lemma}\label{w18}
Let $f\in W^{\beta,2}(\mathbb{R}^{d})$, $\beta>d/2$. Then there exists a constant $k>0$,
which is dindependent of $Y$ and $\delta$ such that to each $f$ there is a band-limited function
$f_{k/q_{Y}}\in \mathcal{B}_{k/q_{Y}}$ with $f_{k/q_{Y}}|Y=f|Y$ and
\begin{equation*}
\|f-f_{k/q_{Y}}\|_{\Psi_{\delta}}\leq 5\|f\|_{\Psi_{\delta}}.
\end{equation*}
\end{lemma}
The last auxiliary result in this section is the well-known sampling inequality
\citep[see][]{nar06,rieger10,s07,w05}.
\begin{lemma}\label{wl3}(\textbf{sampling inequality})
Suppose $\Omega\subseteq \mathbb{R}^{d}$ is a bounded domain with an interior cone condition. Let
$q\in[1,\infty]$ and constants $0\leq \mu<\mu+d/2<\lfloor m \rfloor$. Then there are positive constant $C$, $h_{0}$
such that
\begin{equation}
\|f\|_{W^{\mu,q}(\Omega)}\leq C\left(h_{Y}^{m-\mu-d(1/2-1/q)_{+}}\|f\|_{W^{m,2}(\Omega)}
+h_{Y}^{-\mu}\|\Pi_{Y}f\|_{\infty,Y}\right)
\end{equation}
holds for every discrete set $Y\subset\Omega$ with fill distance at most $h_{Y}\leq h_{0}$ and
every $f\in W^{m,2}(\Omega)$. Here $\Pi_{Y}$ is the function evaluation operator at the data sites $Y$.
\end{lemma}

\section{One-level unsymmetric collocation}
We are now ready to consider the unsymmetric collocation for the equations (\ref{31})-(\ref{32}). According to \citet{s07},
it should be implemented by a nonsquare testing at another data sites with many more degrees
of freedom than the centers $Y=Y^{I}\bigcup Y^{B}$.  On the testing side, we use a set $X=X^{I}\bigcup X^{B}$ with
$X^{I}=\{\textbf{x}_{1},\textbf{x}_{2},\ldots,\textbf{x}_{M_{I}}\}\subseteq \Omega$ and $X^{B}=\{\textbf{x}_{M_{I}+1},\textbf{x}_{M_{I}+2},\ldots,\textbf{y}_{M}\}\subseteq \partial\Omega$.
We use
$$s_{I}:=\sup_{\textbf{x}\in\Omega}\min_{\textbf{x}_{j}\in X^{I}}\|\textbf{x}-\textbf{x}_{j}\|_{2},$$
$$s_{B}:=\max_{1\leq j\leq K}s_{Bj}=\max_{1\leq j\leq K}\left\{\sup_{\textbf{x}\in B=B(\textbf{0},1)}
\min_{\textbf{x}_{j}\in \psi_{j}^{-1}\left(X^{B}\bigcap V_{j}\right)}\|\textbf{x}-\textbf{x}_{j}\|_{2}\right\}$$
to represent the fill distance of the data sites $X^{I}$ and $X^{B}$ respectively. Obviously,
they have the parallel definitions with $h_{I}$ and $h_{B}$ on trial side. Hence choose $s_{X}=\max\{s_{I},s_{B}\}$.
Then we can construct a trial function with the form
$$s_{u}(\textbf{x})=\sum_{j=1}^{N}c_{j}\Phi_{\delta}(\textbf{x}-\textbf{y}_{j}),\quad \textbf{y}_{j}\in Y,$$
which satisfies
\begin{align}
Ls_{u}(\textbf{x}) & = f(\textbf{x}),\quad \textbf{x}\in X^{I},
\label{ls1} \\[4pt]
s_{u}(\textbf{x}) & = g(\textbf{x}), \quad \textbf{x}\in X^{B}.
\label{ls2}
\end{align}
(\ref{ls1})-(\ref{ls2}) is a one-level unsymmetric collocation for the equations (\ref{31})-(\ref{32}),
but with a nonsquare linear systems. We now want to bound the $L^{2}-$error between the solution $u$
and its approximation $s_{u}$.

First we define a map $I_{h}: u\rightarrow I_{h}u\in W_{\delta}$ via interpolation
\begin{equation}\label{interpolation}
I_{h}u(\textbf{y}_{k})=\sum_{j=1}^{N}c_{j}\Phi_{\delta}(\textbf{y}_{k}-\textbf{y}_{j})=u(\textbf{y}_{k}),\quad \textbf{y}_{k}\in Y, k=1,2,\ldots,N,
\end{equation}
for the unknown coefficients $c_{1},c_{2},\ldots,c_{N}$ are determined by solving the above linear system.
Because $I_{h}u$ is a interpolant to $u$ from the trial $W_{\delta}$ at the centers $Y$, so its stability can be obtained
by using Lemma \ref{w18}.
\begin{lemma}\label{wending}
Let $\Omega\subseteq \mathbb{R}^{d}$ be a bounded domain with Lipschitz boundary. Let $Y$ be
quasi-uniform in the sense that $h_{Y}\leq cq_{Y}$. Denote the interpolant of $u\in W^{\sigma,2}(\Omega)$
with $\sigma> d/2+2$ on $Y$ by $I_{h}u$. Then there exists a constant $C>0$ such that
$$\|u-I_{h}u\|_{W^{\sigma,2}(\Omega)}\leq C\delta^{-\sigma}\|u\|_{W^{\sigma,2}(\Omega)}.$$
\end{lemma}
\begin{proof}
From Lemma \ref{w18} combined with Lemma \ref{skuo}, for $u\in W^{\sigma,2}(\Omega)$ there exists a band-limited function
$\widetilde{u}:=(Eu)_{k/q_{Y}}\in \mathcal{B}_{k/q_{Y}}$ such that $u|Y=Eu|Y=\widetilde{u}|Y$. Then
\begin{align*}
\|u-I_{h}u\|_{W^{\sigma,2}(\Omega)}&=\|Eu-I_{h}(Eu)\|_{W^{\sigma,2}(\Omega)}\\[4pt]
&\leq\|Eu-\widetilde{u}\|_{W^{\sigma,2}(\Omega)}+\|\widetilde{u}-I_{h}u\|_{W^{\sigma,2}(\Omega)}\\[4pt]
&\leq\|Eu-\widetilde{u}\|_{W^{\sigma,2}(\Omega)}+\|\widetilde{u}-I_{h}\widetilde{u}\|_{W^{\sigma,2}(\Omega)}.
\end{align*}
The first term on the right can be bounded by
\begin{align*}
\|Eu-\widetilde{u}\|_{W^{\sigma,2}(\Omega)}&\leq\|Eu-\widetilde{u}\|_{W^{\sigma,2}(\mathbb{R}^{d})}\\[4pt]
&\leq C\delta^{-\sigma}\|Eu-\widetilde{u}\|_{\Phi_{\delta}}\leq 5C\delta^{-\sigma}\|Eu\|_{\Phi_{\delta}}\\[4pt]
&\leq5C\delta^{-\sigma}\|Eu\|_{W^{\sigma,2}(\mathbb{R}^{d})}\leq C\delta^{-\sigma}\|u\|_{W^{\sigma,2}(\Omega)}.
\end{align*}
Using Lemma \ref{skuo}, Lemma \ref{mdj} and Lemma \ref{w18}, the upper bound of the second term simply follows from
\begin{align*}
\|\widetilde{u}-I_{h}\widetilde{u}\|_{W^{\sigma,2}(\Omega)}&\leq
C\delta^{-\sigma}\|\widetilde{u}-I_{h}\widetilde{u}\|_{\Phi_{\delta}}\\[4pt]
&\leq C\delta^{-\sigma}\|\widetilde{u}\|_{\Phi_{\delta}}\leq C\delta^{-\sigma}\|\widetilde{u}\|_{W^{\sigma,2}(\mathbb{R}^{d})}\\[4pt]
&\leq C\delta^{-\sigma}\|Eu\|_{W^{\sigma,2}(\mathbb{R}^{d})}\leq C\delta^{-\sigma}\|u\|_{W^{\sigma,2}(\Omega)}.
\end{align*}
\end{proof}

On the other hand, $I_{h}u-s_{u}\in W^{\sigma,2}(\Omega)$ satisfies the equations (\ref{31})-(\ref{32})
with $f=L(I_{h}u-s_{u})$ and $g=I_{h}u-s_{u}$. Hence, with the help of the priori inequality (Lemma \ref{xianyan}) we have
\begin{equation}\label{guji}
\|I_{h}u-s_{u}\|_{L^{2}(\Omega)}\leq C\left\{\|L(I_{h}u-s_{u})\|_{L^{2}(\Omega)}+\|I_{h}u-s_{u}\|_{W^{1/2,2}(\partial\Omega)}\right\}.
\end{equation}
So we need to bound the two terms on the right-hand side separately.
Using the two testing operators $\Pi_{X^{I}}$ and $\Pi_{X^{B}}$ and combining with the sampling inequality, we can
bound $\|L(I_{h}u-s_{u})\|_{L^{2}(\Omega)}$ and $\|I_{h}u-s_{u}\|_{W^{1/2,2}(\partial\Omega)}$ respectively.

\begin{lemma}\label{coerciveness}(\textbf{coerciveness of} $\Pi_{X^{I}},\Pi_{X^{B}}$)
Let $\Omega\subseteq \mathbb{R}^{d}$ be a bounded domain has a $C^{k,s}-$boundary with $s\in[0,1)$. Let $Y$ be
quasi-uniform in the sense that $h_{Y}\leq cq_{Y}$. Let the testing data sites $X$ also be quasi-uniform.
Let $\sigma=k+s$, $k=\lfloor\sigma\rfloor> d/2+2$. Denote the interpolant of $u\in W^{\sigma,2}(\Omega)$
on $Y$ by $I_{h}u$. Let $s_{u}\in W_{\delta}$ be the collocation solution that satisfies
(\ref{ls1})-(\ref{ls2}). Let $s_{I}$ and $s_{B}$ be the fill distance of the data sites $X^{I}$ and $X^{B}$
respectively, and $s_{X}=\max\{s_{I},s_{B}\}$. If the testing discretizations is finer than the trial discretization and
satisfies condition
\begin{equation}\label{condition}
Cs_{X}^{\sigma-2}h_{Y}^{-\sigma}<\frac{1}{2},
\end{equation}
then there exists a constant $C>0$ such that
\begin{equation}\label{462}
\|I_{h}u-s_{u}\|^{2}_{L^{2}(\Omega)}\leq C
 \left\{\|\Pi_{X^{I}}L(I_{h}u-s_{u})\|^{2}_{\infty,X^{I}}
+s_{B}^{-1}\|\Pi_{X^{B}}(I_{h}u-s_{u})\|^{2}_{\infty,X^{B}}\right\}.
\end{equation}
\end{lemma}
\begin{proof}
Using Lemma \ref{skuo}, Lemma \ref{llianxu}, Lemma \ref{invv}, and Lemma \ref{wl3}, we have
\begin{align*}
\|L(I_{h}u-s_{u})\|_{L^{2}(\Omega)}&\leq C\left\{s_{I}^{\sigma-2}\|L(I_{h}u-s_{u})\|_{W^{\sigma-2,2}(\Omega)}
+\|\Pi_{X^{I}}L(I_{h}u-s_{u})\|_{\infty,X^{I}}\right\}
\\[4pt]
&\leq Cs_{I}^{\sigma-2}\|I_{h}u-s_{u}\|_{W^{\sigma,2}(\Omega)}
+C\|\Pi_{X^{I}}L(I_{h}u-s_{u})\|_{\infty,X^{I}}
\\[4pt]
&\leq Cs_{I}^{\sigma-2}\|E(I_{h}u-s_{u})\|_{W^{\sigma,2}(\mathbb{R}^{d})}
+C\|\Pi_{X^{I}}L(I_{h}u-s_{u})\|_{\infty,X^{I}}
\\[4pt]
&\leq Cs_{I}^{\sigma-2}q_{Y}^{-\sigma}\|E(I_{h}u-s_{u})\|_{L^{2}(\mathbb{R}^{d})}
+C\|\Pi_{X^{I}}L(I_{h}u-s_{u})\|_{\infty,X^{I}}
\\[4pt]
&\leq Cs_{I}^{\sigma-2}q_{Y}^{-\sigma}\|I_{h}u-s_{u}\|_{L^{2}(\Omega)}
+C\|\Pi_{X^{I}}L(I_{h}u-s_{u})\|_{\infty,X^{I}}
\\[4pt]
&\leq Cs_{X}^{\sigma-2}h_{Y}^{-\sigma}\|I_{h}u-s_{u}\|_{L^{2}(\Omega)}
+C\|\Pi_{X^{I}}L(I_{h}u-s_{u})\|_{\infty,X^{I}}.
\end{align*}
For the boundary term, we use the definition of the Sobolev norm on the boundary, Lemma \ref{skuo},
Lemma \ref{trace}, and Lemma \ref{wl3}. If we denote
$w_{j}=((I_{h}u-s_{u})\omega_{j})\circ\Psi_{j}\in W^{\sigma-1/2,2}(B)$, then
\begin{align*}
&\quad\|I_{h}u-s_{u}\|^{2}_{W^{1/2,2}(\partial\Omega)}=\sum_{j=1}^{K}\|w_{j}\|^{2}_{W^{1/2,2}(B)}
\\[4pt]
&\leq C\sum_{j=1}^{K}\left(s_{Bj}^{\sigma-1}\|w_{j}\|_{W^{\sigma-1/2,2}(B)}+
s_{Bj}^{-1/2}\|\Pi_{\psi_{j}^{-1}\left(X^{B}\bigcap V_{j}\right)} w_{j}\|_{\infty,\psi_{j}^{-1}\left(X^{B}\bigcap V_{j}\right)}
\right)^{2}
\\[4pt]
&\leq C\sum_{j=1}^{K}\left(s_{Bj}^{2\sigma-2}\|w_{j}\|^{2}_{W^{\sigma-1/2,2}(B)}+
s_{Bj}^{-1}\|\Pi_{\psi_{j}^{-1}\left(X^{B}\bigcap V_{j}\right)} w_{j}\|^{2}_{\infty,\psi_{j}^{-1}\left(X^{B}\bigcap V_{j}\right)}
\right)
\\[4pt]
&\leq Cs_{B}^{2\sigma-2}\|I_{h}u-s_{u}\|^{2}_{W^{\sigma-1/2,2}(\partial\Omega)}+Cs_{B}^{-1}
\|\Pi_{X^{B}} (I_{h}u-s_{u})\|^{2}_{\infty,X^{B}}
\\[4pt]
&\leq Cs_{B}^{2\sigma-2}\|I_{h}u-s_{u}\|^{2}_{W^{\sigma,2}(\Omega)}+C s_{B}^{-1}
\|\Pi_{X^{B}} (I_{h}u-s_{u})\|^{2}_{\infty,X^{B}}
\\[4pt]
&\leq Cs_{X}^{2\sigma-2}h_{Y}^{-2\sigma}\|I_{h}u-s_{u}\|^{2}_{L^{2}(\Omega)}+C s_{B}^{-1}
\|\Pi_{X^{B}} (I_{h}u-s_{u})\|^{2}_{\infty,X^{B}},
\end{align*}
where the last inequality is obtained by the same as the last four steps for the term $\|L(I_{h}u-s_{u})\|_{L^{2}(\Omega)}$.

From the inequality (\ref{guji}) and the condition (\ref{condition}), we have
\begin{align*}
\|I_{h}u-s_{u}\|^{2}_{L^{2}(\Omega)}&\leq
C\|L(I_{h}u-s_{u})\|^{2}_{L^{2}(\Omega)}+C\|I_{h}u-s_{u}\|^{2}_{W^{1/2,2}(\partial\Omega)}
\\[4pt]
&\leq \frac{1}{2}\|I_{h}u-s_{u}\|^{2}_{L^{2}(\Omega)}+C\|\Pi_{X^{I}}L(I_{h}u-s_{u})\|^{2}_{\infty,X^{I}}
+Cs_{B}^{-1}
\|\Pi_{X^{B}} (I_{h}u-s_{u})\|^{2}_{\infty,X^{B}}.
\end{align*}
This yields the result (\ref{462}).
\end{proof}

If we assume the linear systems (\ref{ls1})-(\ref{ls2}) can be solved by $s_{u}\in W_{\delta}$ with the tolerance
\begin{equation}\label{tolerance}
\|\Pi_{X^{I}}L(u-s_{u})\|^{2}_{\infty,X^{I}}+s_{B}^{-1}\|\Pi_{X^{B}}(u-s_{u})\|^{2}_{\infty,X^{B}}
\leq C\delta^{-2\sigma}h_{Y}^{2(\sigma-2-d/2)}\|u\|^{2}_{W^{\sigma,2}(\Omega)},
\end{equation}
then we can obtain the boundedness of $\Pi_{X^{I}},\Pi_{X^{B}}$.
\begin{lemma}\label{boundedness}(\textbf{boundedness of} $\Pi_{X^{I}},\Pi_{X^{B}}$)
Under the same assumption of Lemma \ref{coerciveness}. Suppose one of the linear solvers
for the equations (\ref{ls1})-(\ref{ls2}) satisfies (\ref{tolerance}).
Then there exists a constant $C>0$ such that
\begin{equation}
\|\Pi_{X^{I}}L(I_{h}u-s_{u})\|^{2}_{\infty,X^{I}}
+s_{B}^{-1}\|\Pi_{X^{B}}(I_{h}u-s_{u})\|^{2}_{\infty,X^{B}}\leq
C\delta^{-2\sigma}h^{2(\sigma-2-d/2)}_{Y}\|u\|^{2}_{W^{\sigma,2}(\Omega)}.
\end{equation}
\end{lemma}
\begin{proof}
From Lemma \ref{wl3}, Lemma \ref{llianxu}, Lemma \ref{wending}, we can bound
\begin{align*}
\|\Pi_{X^{I}}L(I_{h}u-s_{u})\|^{2}_{\infty,X^{I}}&\leq
2\|\Pi_{X^{I}}L(I_{h}u-u)\|^{2}_{\infty,X^{I}}+2\|\Pi_{X^{I}}L(u-s_{u})\|^{2}_{\infty,X^{I}}
\\[4pt]
&=
2\|L(I_{h}u-u)\|^{2}_{\infty,X^{I}}+2\|\Pi_{X^{I}}L(u-s_{u})\|^{2}_{\infty,X^{I}}
\\[4pt]
&\leq
2\|L(I_{h}u-u)\|^{2}_{\infty,\Omega}+2\|\Pi_{X^{I}}L(u-s_{u})\|^{2}_{\infty,X^{I}}
\\[4pt]
&\leq
Ch_{I}^{2(\sigma-2-d/2)}\|L(I_{h}u-u)\|^{2}_{W^{\sigma-2,2}(\Omega)}+2\|\Pi_{X^{I}}L(u-s_{u})\|^{2}_{\infty,X^{I}}
\\[4pt]
&\leq
Ch_{I}^{2(\sigma-2-d/2)}\|I_{h}u-u\|^{2}_{W^{\sigma,2}(\Omega)}+2\|\Pi_{X^{I}}L(u-s_{u})\|^{2}_{\infty,X^{I}}
\\[4pt]
&\leq
C\delta^{-2\sigma}h_{I}^{2(\sigma-2-d/2)}\|u\|^{2}_{W^{\sigma,2}(\Omega)}+2\|\Pi_{X^{I}}L(u-s_{u})\|^{2}_{\infty,X^{I}}.
\end{align*}
To bounding the boundary term,
we denote $w_{j}=((I_{h}u-u)\omega_{j})\circ\Psi_{j}\in W^{\sigma-1/2,2}(B)$. Then
we can use the sampling inequality on the unit ball $B=B(\textbf{0},1)\subseteq \mathbb{R}^{d-1}$.
Using the sampling inequality, the definition of the Sobolev norm on the boundary, the trace theorem, and Lemma \ref{wending},
then we have
\begin{align*}
\|\Pi_{X^{B}}(I_{h}u-s_{u})\|^{2}_{\infty,X^{B}}&\leq
2\|\Pi_{X^{B}}(I_{h}u-u)\|^{2}_{\infty,X^{B}}+2\|\Pi_{X^{B}}(u-s_{u})\|^{2}_{\infty,X^{B}}
\\[4pt]
&=
2\|I_{h}u-u\|^{2}_{\infty,X^{B}}+2\|\Pi_{X^{B}}(u-s_{u})\|^{2}_{\infty,X^{B}}
\\[4pt]
&\leq
2\|I_{h}u-u\|^{2}_{\infty,\partial\Omega}+2\|\Pi_{X^{B}}(u-s_{u})\|^{2}_{\infty,X^{B}}
\\[4pt]
&\leq
C\max_{1\leq j\leq K}\|w_{j}\|^{2}_{L^{\infty}(B)}+2\|\Pi_{X^{B}}(u-s_{u})\|^{2}_{\infty,X^{B}}
\\[4pt]
&\leq
C\max_{1\leq j\leq K}h_{Bj}^{2(\sigma-1/2-d/2)}\|w_{j}\|^{2}_{W^{\sigma-1/2,2}(B)}+2\|\Pi_{X^{B}}(u-s_{u})\|^{2}_{\infty,X^{B}}
\\[4pt]
&\leq
C h_{B}^{2(\sigma-1/2-d/2)}\sum_{j=1}^{K}\|w_{j}\|^{2}_{W^{\sigma-1/2,2}(B)}+2\|\Pi_{X^{B}}(u-s_{u})\|^{2}_{\infty,X^{B}}
\\[4pt]
&=
C h_{B}^{2(\sigma-1/2-d/2)}\|I_{h}u-u\|^{2}_{W^{\sigma-1/2,2}(\partial\Omega)}+2\|\Pi_{X^{B}}(u-s_{u})\|^{2}_{\infty,X^{B}}
\\[4pt]
&\leq
C h_{B}^{2(\sigma-1/2-d/2)}\|I_{h}u-u\|^{2}_{W^{\sigma,2}(\Omega)}+2\|\Pi_{X^{B}}(u-s_{u})\|^{2}_{\infty,X^{B}}
\\[4pt]
&\leq
C \delta^{-2\sigma}h_{B}^{2(\sigma-1/2-d/2)}\|u\|^{2}_{W^{\sigma,2}(\Omega)}+2\|\Pi_{X^{B}}(u-s_{u})\|^{2}_{\infty,X^{B}}.
\end{align*}
We now multiply the boundary inequality by a factor $s_{B}^{-1}$ and combine it with the domain inequality.
By using the condition $\lfloor\sigma\rfloor> d/2+2$, the condition (\ref{condition}),
the quasi-uniform property of $X$, the quasi-uniform property of $Y$, and the condition (\ref{tolerance}), we have
\begin{align*}
&\quad\|\Pi_{X^{I}}L(I_{h}u-s_{u})\|^{2}_{\infty,X^{I}}+s_{B}^{-1}\|\Pi_{X^{B}}(I_{h}u-s_{u})\|^{2}_{\infty,X^{B}}
\\[4pt]
&\leq
C\delta^{-2\sigma}\left(h_{I}^{2(\sigma-2-d/2)}+s_{B}^{-1}h_{B}^{2(\sigma-1/2-d/2)}\right)\|u\|^{2}_{W^{\sigma,2}(\Omega)}
\\[4pt]
&\leq
C\delta^{-2\sigma}\left(h_{Y}^{2(\sigma-2-d/2)}+h_{Y}^{-\sigma/(\sigma-2)}h_{Y}^{2(\sigma-1/2-d/2)}\right)\|u\|^{2}_{W^{\sigma,2}(\Omega)}
\\[4pt]
&\leq
C\delta^{-2\sigma}h^{2(\sigma-2-d/2)}_{Y}\|u\|^{2}_{W^{\sigma,2}(\Omega)}.
\end{align*}
\end{proof}

The next theorem is our main convergence result of
the one-level unsymmetric collocation for the equations (\ref{31})-(\ref{32}).
\begin{theorem}\label{onelevel}(\textbf{Convergence of one-level unsymmetric collocation})
Let $\Omega\subseteq \mathbb{R}^{d}$ be a bounded domain has a $C^{k,s}-$boundary with $s\in[0,1)$.
Let the trial data $Y$ with the fill distance $h_{Y}=\max\{h_{I},h_{B}\}$
and the testing data $X$ with the
fill distance $s_{X}=\max\{s_{I},s_{B}\}$ both be quasi-uniform.
Let $\sigma=k+s$, $k=\lfloor\sigma\rfloor> d/2+2$. Let $u\in W^{\sigma,2}(\Omega)$ be the solution of
(\ref{31})-(\ref{32}). Suppose $s_{u}\in W_{\delta}$ is the
collocation solution that satisfies (\ref{ls1})-(\ref{ls2}).
Assume that (\ref{ls1})-(\ref{ls2}) can be solved by $s_{u}$ with the tolerance (\ref{tolerance}).
If the testing discretizations is finer than the trial discretization and
satisfies condition (\ref{condition}),
then there exists a constant $C>0$ such that
\begin{equation*}\label{46}
\|u-s_{u}\|_{L^{2}(\Omega)}\leq
C\delta^{-\sigma}h^{\sigma-2-d/2}_{Y}\|u\|_{W^{\sigma,2}(\Omega)}.
\end{equation*}
\end{theorem}
\begin{proof}
Using Lemma \ref{wl3} and combining Lemma \ref{coerciveness} with Lemma \ref{boundedness}, we have
\begin{align*}
\|u-s_{u}\|_{L^{2}(\Omega)}&\leq \|u-I_{h}u\|_{L^{2}(\Omega)}+\|I_{h}u-s_{u}\|_{L^{2}(\Omega)}
\\[4pt]
&\leq Ch_{Y}^{\sigma}\|u-I_{h}u\|_{W^{\sigma,2}(\Omega)}+
C\delta^{-\sigma}h^{\sigma-2-d/2}_{Y}\|u\|_{W^{\sigma,2}(\Omega)}
\\[4pt]
&\leq C\delta^{-\sigma}h_{Y}^{\sigma}\|u\|_{W^{\sigma,2}(\Omega)}+
C\delta^{-\sigma}h^{\sigma-2-d/2}_{Y}\|u\|_{W^{\sigma,2}(\Omega)}
\\[4pt]
&\leq C\delta^{-\sigma}h^{\sigma-2-d/2}_{Y}\|u\|_{W^{\sigma,2}(\Omega)}.
\end{align*}
\end{proof}
\section{Multilevel unsymmetric collocation}
We now consider implementing the unsymmetric collocation for
(\ref{31})-(\ref{32}) on the increasingly dense data sites.
Assume that we are given a sequence of interior centers $Y^{I}_{1},Y^{I}_{2},\ldots\subseteq \Omega$
and boundary centers $Y^{B}_{1},Y^{B}_{2},\ldots\subseteq \partial\Omega$. Corresponding, we are also
given a sequence of interior testing data sites $X^{I}_{1},X^{I}_{2},\ldots\subseteq \Omega$
and boundary testing data sites $X^{B}_{1},X^{B}_{2},\ldots\subseteq \partial\Omega$.
We denote
$$Y_{j}=Y_{j}^{I}\bigcup Y_{j}^{B},\quad X_{j}=X_{j}^{I}\bigcup X_{j}^{B}.$$
We denote the fill distance of $Y_{j}^{I}$ and $X_{j}^{I}$ by $\tilde{h}_{j},\tilde{s}_{j}$, and the
fill distance of $Y_{j}^{B}$ and $X_{j}^{B}$ by $\tilde{\tilde{h}}_{j},\tilde{\tilde{s}}_{j}$ respectively.
Let $h_{j}=\max\{\tilde{h}_{j},\tilde{\tilde{h}}_{j}\},s_{j}=\max\{\tilde{s}_{j},\tilde{\tilde{s}}_{j}\}$
be the fill distance of $Y_{j}$ and $X_{j}$ respectively, which satisfy
\begin{equation}\label{condition2}
Cs_{j}^{\sigma-2}h_{j}^{-\sigma}<\frac{1}{2}.
\end{equation}
Then, with a radial function $\Phi$ satisfied (\ref{adc}) and a series of decreasing scaling parameters
$\delta_{1}>\delta_{2}>\ldots>\delta_{j}>\ldots\in(0,1]$, we can define
\begin{equation}\label{rbf2}
\Phi_{j}:=\Phi_{\delta_{j}}=\delta_{j}^{-d}\Phi((\textbf{x}-\textbf{y})/\delta_{j}), \quad \textbf{x},\textbf{y}\in\mathbb{R}^{d}.
\end{equation}
Using the radial functions $\Phi_{j}$ and the trial centers $Y_{j}$, the one-level kernel-based trial spaces can be constructed as
\begin{equation}
\label{tss}
W_{j}=\textrm{span}\{\Phi_{j}(\cdot-\textbf{y})|\textbf{y}\in Y_{j}\}.
\end{equation}
Let $v_{j}\in W_{j}$ be some trial functions which have the forms
\begin{equation*}
v_{j}(\textbf{x})=\sum_{j=1}^{\sharp(Y_{j})}c_{j}\Phi_{j}(\textbf{x}-\textbf{y}_{j}),\quad \textbf{y}_{j}\in Y_{j},
\end{equation*}
then a multilevel unsymmetric collocation algorithm for the equations (\ref{31})-(\ref{32}) can be obtained as the following form.

\begin{algorithm}\label{algo}(\textbf{multilevel unsymmetric collocation})

Given the right-hand sides $f$ and $g$, we implement

(1) Set $u_{0}=0,f_{0}=f,g_{0}=g$

(2) for $j=1,2,\ldots$ do

\quad\quad $\bullet$ Determine the local correction $v_{j}$ to $f_{j-1}$ and $g_{j-1}$ with
\begin{align}
Lv_{j}(\textbf{x})&=f_{j-1}(\textbf{x}),\quad \textbf{x}\in X_{j}^{I},
\label{loc1}
\\[4pt]
v_{j}(\textbf{x})&=g_{j-1}(\textbf{x}),\quad \textbf{x}\in X_{j}^{B},
\label{loc2}
\end{align}

\quad\quad $\bullet$ Update the global approximation and the residuals
\begin{align*}
u_{j}&=u_{j-1}+v_{j},
\\[4pt]
f_{j}&=f_{j-1}-Lv_{j},
\\[4pt]
g_{j}&=g_{j-1}-v_{j}.
\end{align*}
\end{algorithm}
Suppose $u\in W^{\sigma,2}(\Omega)$ is the exact solution of the equations (\ref{31})-(\ref{32}),
we denote $e_{j}=u-u_{j}$ be the error at level $j$. Then Algorithm \ref{algo} satisfies the following relations:
\begin{equation*}
u_{j}=\sum_{k=1}^{j}v_{k},\quad e_{j}=u-u_{j}=e_{j-1}-v_{j},\quad f_{j}=f-Lu_{j},\quad \textrm{and}\quad g_{j}=g-u_{j}.
\end{equation*}
\begin{lemma}\label{51}
Let $\Omega\subseteq \mathbb{R}^{d}$ be a bounded domain has a $C^{k,s}-$boundary with $s\in[0,1)$.
Let the trial data $Y_{j}$ with the fill distance $h_{j}$
and the testing data $X_{j}$ with the
fill distance $s_{j}$ both be quasi-uniform for $j=1,2,\ldots$.
Let $\sigma=k+s$, $k=\lfloor\sigma\rfloor> d/2+2$.
Denote the interpolant of $e_{j-1}\in W^{\sigma,2}(\Omega)$
on $Y_{j}$ by $I_{j}: e_{j-1}\rightarrow I_{j}e_{j-1}\in W_{j}$.
Suppose $v_{j}\in W_{j}$ are the
local correction decided by Algorithm \ref{algo}.
Assume that (\ref{loc1})-(\ref{loc2}) can be solved by $v_{j}$ with the tolerance
\begin{equation}\label{tolerance2}
\|\Pi_{X^{I}_{j}}Le_{j}\|^{2}_{\infty,X^{I}_{j}}+\tilde{\tilde{s}}_{j}^{-1}\|\Pi_{X^{B}_{j}}e_{j}\|^{2}_{\infty,X^{B}_{j}}
\leq C\delta_{j}^{-2\sigma}h_{j}^{2(\sigma-2-d/2)}\|e_{j-1}\|^{2}_{W^{\sigma,2}(\Omega)},
\end{equation}
for $j=1,2,\ldots$.
If the testing discretizations is finer than the trial discretization and
satisfies condition (\ref{condition2}),
then there exists a constant $C>0$ such that
\begin{equation}\label{yichuan1}
\|e_{j}\|_{W^{\sigma,2}(\Omega)}\leq
C\delta_{j}^{-\sigma}h_{j}^{-(2+d/2)}\|e_{j-1}\|_{W^{\sigma,2}(\Omega)}.
\end{equation}
\end{lemma}
\begin{proof}
From Lemma \ref{wending}, Lemma \ref{skuo}, Lemma \ref{invv},
Lemma \ref{coerciveness}, and Lemma \ref{boundedness}, we obtain
\begin{align*}
\|e_{j}\|_{W^{\sigma,2}(\Omega)}&=\|e_{j-1}-v_{j}\|_{W^{\sigma,2}(\Omega)}
\\[4pt]
&\leq\|e_{j-1}-I_{j}e_{j-1}\|_{W^{\sigma,2}(\Omega)}+\|I_{j}e_{j-1}-v_{j}\|_{W^{\sigma,2}(\Omega)}
\\[4pt]
&\leq C\delta_{j}^{-\sigma}\|e_{j-1}\|_{W^{\sigma,2}(\Omega)}+\|E(I_{j}e_{j-1}-v_{j})\|_{W^{\sigma,2}(\mathbb{R}^{d})}
\\[4pt]
&\leq C\delta_{j}^{-\sigma}\|e_{j-1}\|_{W^{\sigma,2}(\Omega)}+Ch_{j}^{-\sigma}\|E(I_{j}e_{j-1}-v_{j})\|_{L^{2}(\mathbb{R}^{d})}
\\[4pt]
&\leq C\delta_{j}^{-\sigma}\|e_{j-1}\|_{W^{\sigma,2}(\Omega)}+Ch_{j}^{-\sigma}\|I_{j}e_{j-1}-v_{j}\|_{L^{2}(\Omega)}
\\[4pt]
&\leq C\delta_{j}^{-\sigma}\|e_{j-1}\|_{W^{\sigma,2}(\Omega)}+C\delta_{j}^{-\sigma}h_{j}^{-\sigma}
h_{j}^{\sigma-2-d/2}\|e_{j-1}\|_{W^{\sigma,2}(\Omega)}
\\[4pt]
&\leq C\delta_{j}^{-\sigma}h_{j}^{-(2+d/2)}\|e_{j-1}\|_{W^{\sigma,2}(\Omega)}.
\end{align*}
\end{proof}

We are now in the position to prove convergence of the multilevel unsymmetric
collocation for the equations (\ref{31})-(\ref{32}).
\begin{theorem}\label{multilevel}(\textbf{Convergence of multilevel unsymmetric collocation})
Let $\Omega\subseteq \mathbb{R}^{d}$ be a bounded domain has a $C^{k,s}-$boundary with $s\in[0,1)$.
Let the trial data $Y_{j}$ with the fill distance $h_{j}$
and the testing data $X_{j}$ with the
fill distance $s_{j}$ both be quasi-uniform for $j=1,2,\ldots$.
Let $\sigma=k+s$, $k=\lfloor\sigma\rfloor> d/2+2$.
Let $u\in W^{\sigma,2}(\Omega)$ be the solution of
(\ref{31})-(\ref{32}). Suppose $v_{j}\in W_{j}$ are the
local correction decided by Algorithm \ref{algo}.
Assume that (\ref{loc1})-(\ref{loc2}) can be solved by $v_{j}$ with the tolerance
(\ref{tolerance2}). Assume that $\delta_{j}$, $h_{j}$ satisfy the relations:
\begin{equation}\label{restrit}
\delta_{j}=h_{j-1},\quad h_{j}=\mu h_{j-1}^{\frac{2\sigma}{\sigma-2-d/2}},\quad j=1,2,\ldots.
\end{equation}
with appropriate $h_{0}$ and $0<\mu<1$.
And assume the constant $\mu$ has been chosen sufficiently small such that
$$\alpha:=C\mu^{\sigma-2-d/2}<1.$$
If the testing discretizations is finer than the trial discretization and
satisfies condition (\ref{condition2}),
then there exists a constant $C>0$ such that
\begin{equation*}
\|u-u_{k}\|_{L^{2}(\Omega)}\leq C\alpha^{k}\|u\|_{W^{\sigma,2}(\Omega)}.
\end{equation*}
\end{theorem}
\begin{proof}
First, we establish the error genetic relationship from $j$ level to $j-1$ level.
The technical process is similar to the proofs in \citet{ch14,farell13,w10,w18}.
\begin{align*}
\|Ee_{j}\|^{2}_{\Phi_{j+1}}&=\int_{\mathbb{R}^{d}}\frac{|\widehat{Ee_{j}}(\boldsymbol{\omega})|^{2}}
{\widehat{\Phi_{j+1}}(\boldsymbol{\omega})}d\boldsymbol{\omega}
=\int_{\mathbb{R}^{d}}\frac{|\widehat{Ee_{j}}(\boldsymbol{\omega})|^{2}}
{\widehat{\Phi}(\delta_{j+1}\boldsymbol{\omega})}d\boldsymbol{\omega}
\\[4pt]
&\leq C\int_{\mathbb{R}^{d}}|\widehat{Ee_{j}}(\boldsymbol{\omega})|^{2}
(1+\delta_{j+1}^{2}\|\boldsymbol{\omega}\|_{2}^{2})^{\sigma}d\boldsymbol{\omega}
\\[4pt]
&\leq C\int_{\|\boldsymbol{\omega}\|_{2}\leq\frac{1}{\delta_{j+1}}}|\widehat{Ee_{j}}(\boldsymbol{\omega})|^{2}
(1+\delta_{j+1}^{2}\|\boldsymbol{\omega}\|_{2}^{2})^{\sigma}d\boldsymbol{\omega}
\\[4pt]
&\quad +C\int_{\|\boldsymbol{\omega}\|_{2}\geq\frac{1}{\delta_{j+1}}}|\widehat{Ee_{j}}(\boldsymbol{\omega})|^{2}
(1+\delta_{j+1}^{2}\|\boldsymbol{\omega}\|_{2}^{2})^{\sigma}d\boldsymbol{\omega}
\\[4pt]
&=: CI_{1}+CI_{2}.
\end{align*}
For the first term on the right-hand side, we use Lemma \ref{skuo}, the convergence of one-level
unsymmetric collocation, Lemma \ref{mdj}, and the variable relationship
\begin{equation*}
\delta^{-4\sigma}_{j}h^{2(\sigma-2-d/2)}_{j}=h_{j-1}^{-4\sigma}h_{j}^{2(\sigma-2-d/2)}
=\mu^{2(\sigma-2-d/2)},
\end{equation*}
we obtain
\begin{align*}
I_{1}&\leq2^{\sigma}\int_{\|\boldsymbol{\omega}\|_{2}\leq\frac{1}{\delta_{j+1}}}
|\widehat{Ee_{j}}(\boldsymbol{\omega})|^{2}d\boldsymbol{\omega}
\leq2^{\sigma}\int_{\mathbb{R}^{d}}
|\widehat{Ee_{j}}(\boldsymbol{\omega})|^{2}d\boldsymbol{\omega}
\\[4pt]
&= 2^{\sigma}\|Ee_{j}\|^{2}_{L^{2}(\mathbb{R}^{d})}
\leq C\|e_{j}\|^{2}_{L^{2}(\Omega)}
\\[4pt]
&\leq C\delta^{-2\sigma}_{j}h^{2(\sigma-2-d/2)}_{j}\|e_{j-1}\|^{2}_{W^{\sigma,2}(\Omega)}
\\[4pt]
&\leq C\delta^{-2\sigma}_{j}h^{2(\sigma-2-d/2)}_{j}\|Ee_{j-1}\|^{2}_{W^{\sigma,2}(\mathbb{R}^{d})}
\\[4pt]
&\leq C\delta^{-4\sigma}_{j}h^{2(\sigma-2-d/2)}_{j}\|Ee_{j-1}\|^{2}_{\Phi_{j}}
\\[4pt]
&\leq C\mu^{2(\sigma-2-d/2)}\|Ee_{j-1}\|^{2}_{\Phi_{j}}
\end{align*}
For the second term on the right-hand side, we use Lemma \ref{skuo}, Lemma \ref{51}, Lemma \ref{mdj},
and the variable relationship
\begin{equation*}
\delta_{j+1}^{2\sigma}\delta_{j}^{-4\sigma}h_{j}^{-(4+d)}=h_{j}^{2\sigma}h_{j-1}^{-4\sigma}h_{j}^{-(4+d)}
=\mu^{2(\sigma-2-d/2)},
\end{equation*}
we obtain
\begin{align*}
I_{2}&\leq 2^{\sigma}\delta_{j+1}^{2\sigma}\int_{\|\boldsymbol{\omega}\|_{2}
\geq\frac{1}{\delta_{j+1}}}|\widehat{Ee_{j}}(\boldsymbol{\omega})|^{2}
\|\boldsymbol{\omega}\|_{2}^{2\sigma}d\boldsymbol{\omega}
\\[4pt]
&\leq 2^{\sigma}\delta_{j+1}^{2\sigma}\int_{\mathbb{R}^{d}}|\widehat{Ee_{j}}(\boldsymbol{\omega})|^{2}
(1+\|\boldsymbol{\omega}\|_{2}^{2})^{\sigma}d\boldsymbol{\omega}
\\[4pt]
&\leq2^{\sigma}\delta_{j+1}^{2\sigma}\|Ee_{j}\|^{2}_{W^{\sigma,2}(\mathbb{R}^{d})}
\\[4pt]
&\leq C\delta_{j+1}^{2\sigma}\|e_{j}\|^{2}_{W^{\sigma,2}(\Omega)}
\\[4pt]
&\leq C\delta_{j+1}^{2\sigma}\delta_{j}^{-2\sigma}h_{j}^{-(4+d)}\|e_{j-1}\|^{2}_{W^{\sigma,2}(\Omega)}
\\[4pt]
&\leq C\delta_{j+1}^{2\sigma}\delta_{j}^{-4\sigma}h_{j}^{-(4+d)}\|Ee_{j-1}\|^{2}_{\Phi_{j}}
\\[4pt]
&\leq C\mu^{2(\sigma-2-d/2)}\|Ee_{j-1}\|^{2}_{\Phi_{j}}.
\end{align*}
From the upper bounds of $I_{1},I_{2}$, choosing $\alpha=C\mu^{\sigma-2-d/2}$ we have
\begin{equation}\label{yichuan}
\|Ee_{j}\|_{\Phi_{j+1}}\leq C\mu^{\sigma-2-d/2}\|Ee_{j-1}\|_{\Phi_{j}}=\alpha\|Ee_{j-1}\|_{\Phi_{j}}.
\end{equation}
Then we obtain
\begin{align*}
\|u-u_{k}\|_{L^{2}(\Omega)}&=\|e_{k}\|_{L^{2}(\Omega)}=\|e_{k-1}-v_{k}\|_{L^{2}(\Omega)}
\\[4pt]
&\leq\|e_{k-1}-I_{k}e_{k-1}\|_{L^{2}(\Omega)}+\|I_{k}e_{k-1}-v_{k}\|_{L^{2}(\Omega)}
\\[4pt]
&\leq C\delta_{k}^{-\sigma}h_{k}^{\sigma}\|e_{k-1}\|_{W^{\sigma,2}(\Omega)}+
C\delta_{k}^{-\sigma}h_{k}^{\sigma-2-d/2}\|e_{k-1}\|_{W^{\sigma,2}(\Omega)}
\\[4pt]
&\leq C\delta_{k}^{-\sigma}h_{k}^{\sigma-2-d/2}\|e_{k-1}\|_{W^{\sigma,2}(\Omega)}
\\[4pt]
&\leq C\delta_{k}^{-\sigma}h_{k}^{\sigma-2-d/2}\|Ee_{k-1}\|_{W^{\sigma,2}(\mathbb{R}^{d})}
\\[4pt]
&\leq C\delta_{k}^{-2\sigma}h_{k}^{\sigma-2-d/2}\|Ee_{k-1}\|_{\Phi_{k}}
\\[4pt]
&\leq C\alpha\|Ee_{k-1}\|_{\Phi_{k}}.
\end{align*}
Then using (\ref{yichuan}) $k$ times, we have
\begin{equation*}
\|u-u_{k}\|_{L^{2}(\Omega)}\leq C\alpha\|Ee_{k-1}\|_{\Phi_{k}}\leq\ldots\leq C\alpha^{k}\|Eu\|_{\Phi_{1}}
\leq C\alpha^{k}\|u\|_{W^{\sigma,2}(\Omega)}.
\end{equation*}
\end{proof}

\section{Numerical example}
\subsection{Computational aspects}
By using the two testing operators $\Pi_{X^{I}}$ and $\Pi_{X^{B}}$, the linear systems (\ref{ls1})-(\ref{ls2}) can be rewritten as
\begin{align}
\Pi_{X^{I}}Ls_{u}& = \Pi_{X^{I}}f = \Pi_{X^{I}}Lu,
\label{61} \\[4pt]
\Pi_{X^{B}}s_{u} & = \Pi_{X^{B}}g =\Pi_{X^{B}}u.
\label{62}
\end{align}
We denote the residual as
$r=\left(
  \begin{array}{c}
    r_{I}\\
    r_{B} \\
  \end{array}
\right)$, with
$$r_{I}=\Pi_{X^{I}}L(u-s_{u}),\quad r_{B}=\Pi_{X^{B}}(u-s_{u}).$$
To ensure the assumption (\ref{tolerance}), we need to bound $r_{I}$ and $r_{B}$ respectively. That is
\begin{equation*}
\|r_{I}\|_{\infty,X^{I}}
\leq C\delta^{-\sigma}h_{Y}^{\sigma-2-d/2}\|u\|_{W^{\sigma,2}(\Omega)},
\end{equation*}
\begin{align*}
\|r_{B}\|_{\infty,X^{I}}
&\leq Cs_{B}^{1/2}\delta^{-\sigma}h_{Y}^{\sigma-2-d/2}\|u\|_{W^{\sigma,2}(\Omega)}\\[4pt]
&\leq Cs_{X}^{1/2}\delta^{-\sigma}h_{Y}^{\sigma-2-d/2}\|u\|_{W^{\sigma,2}(\Omega)}\\[4pt]
&\leq C\delta^{-\sigma}h_{Y}^{\sigma/(2\sigma-4)+\sigma-2-d/2}\|u\|_{W^{\sigma,2}(\Omega)},
\end{align*}
where we used the condition (\ref{condition}) for the proof of the second inequality. Thus, we only
ask the residual to meet
\begin{equation}\label{residual}
\|r\|_{\infty,X}\leq  C\delta^{-\sigma}h_{Y}^{\sigma/(2\sigma-4)+\sigma-2-d/2}\|u\|_{W^{\sigma,2}(\Omega)}.
\end{equation}
The condition (\ref{tolerance2}) is the same as (\ref{tolerance}), but it needs to be satisfied at level $j$.

We consider the following Poisson problem with Dirichlet boundary conditions:
\begin{eqnarray*}
\triangle u(x,y)&=&-\frac{5}{4}\pi^{2}\sin(\pi x)\cos(\frac{\pi y}{2}), \quad (x,y)\in\Omega=(0,1)^{2},\nonumber\\
 u(x,y)&=&\sin(\pi x), \quad (x,y)\in\Gamma_{1},\nonumber\\
 u(x,y)&=&0, \quad (x,y)\in\Gamma_{2},
 \end{eqnarray*}
where $\Gamma_{1}=\{(x,y): 0\leq x\leq1, y=0\}$ and $\Gamma_{2}=\partial\Omega \setminus \Gamma_{1}$.
The exact solution of the problem is given by
$$u(x,y)=\sin(\pi x)\cos(\frac{\pi y}{2}).$$

In our experiments, Wendland's compactly supported function
$\phi(r)=(1-r)_{+}^{8}(32r^{3}+25r^{2}+8r+1)\in C^{6}(\mathbb{R}^{2})$ is used to construct trial spaces.
This function can generate Sobolev space $W^{\sigma,2}(\mathbb{R}^{2})$ with $\sigma=4.5$.
For the every non square linear systems
$$Ac=b$$
produced by RBF discretization, we use the conjugate gradient method (CG) to solve it
by multiplying the left and right sides by $A^{T}$. Simple calculation shows that the tolerance of CG is about $h_{Y}^{2.4}$.

\subsection{Numerical results}
Although \citet{f07} has provided some numerical examples to show the performance of one-level and multilevel unsymmetric collocation, 
we will restrict ourselves here to  nonsquare collocation because of \citet{hon}. Computations were performed on a laptop with 1.3 GHz Intel Core i7-1065G7 processor, using MATLAB running in the Windows 10 operating system. 

We choose the point sets to be on nested uniform grids and create additional equally spaced $4(\sqrt{N}-1)$ boundary points. To ensure the conditions (\ref{condition}) or (\ref{condition2}), the collocation is implemented on the next fine point sets. For the one-level collocation case, we fix $\delta=2.0$. For the multilevel case, by using the conditions of the Theorem \ref{multilevel}
we use a variable
$$\delta_{j}=v\left(\frac{h_{j}}{\mu}\right)^{\frac{\sigma-2-d/2}{2\sigma}},$$
but let the meshsize have a linear relationship $h_{j}=\mu h_{j-1}, \mu=0.5$ as the same as the one-level case.
In order to be consistent with the experiments in \citet{farell13}, we select $v=2.4$. 

The numerical results are shown in Tables 1,2, where we have given
$L^{2}(\Omega)$ error discretized on a $10000\times10000$ grid, the producing order, the tolerance determined by
(\ref{tolerance}) or (\ref{tolerance2}), and the number of CG iterations. It seems that we can get the expected approximation order.
These results differ from Table 6.2 in \citet{farell13} and Table 41.3 in \citet{f07}, because we are implementing nonsquare collocation under 
the conditions (\ref{condition}) or (\ref{condition2}), and the tolerance (\ref{tolerance}) or (\ref{tolerance2}).
\begin{table}[!h]
\tabcolsep 5mm \caption{One-level unsymmetric collocation: fixed $\delta=2.0$, expected order $\approx 1.5$.}
\begin{center}
\begin{tabular}{r@{}lr@{}lr@{}lr@{.}lr@{}lr@{.}lr@{}l}
\hline\multicolumn{2}{c}{Level}& \multicolumn{2}{c}{$\delta$}&\multicolumn{2}{c}{$N$}&\multicolumn{2}{c}{$L^{2}(\Omega)$}&\multicolumn{2}{c}{Order}
&\multicolumn{2}{c}{Tolerance}
&\multicolumn{2}{c}{CG}
\\ \hline
1&&2.0&&9&&      4&936e-01&    &&        0&1895&  6\\
2&&2.0&&25&&     3&623e-01&  0.446&&  0&0359&  20\\
3&&2.0&&81&&     9&849e-02&  1.879&&  0&0068&  297\\
4&&2.0&&289&&    2&830e-02&  1.799&&  0&0013&  8241\\
\hline
\end{tabular}
\end{center}
\end{table}

\textcolor[rgb]{1.00,0.00,0.00}{\begin{table}[!h]
\tabcolsep 5mm \caption{Multilevel unsymmetric collocation: variable $\delta_{j}=v\left(\frac{h_{j}}{\mu}\right)^{\frac{\sigma-2-d/2}{2\sigma}}$,
$v=2.4$, expected order $\approx 1.5$.}
\begin{center}
\begin{tabular}{r@{}lr@{}lr@{}lr@{.}lr@{}lr@{.}lr@{}l}
\hline\multicolumn{2}{c}{Level}& \multicolumn{2}{c}{$\delta_{j}$}&\multicolumn{2}{c}{$N$}&\multicolumn{2}{c}{$L^{2}(\Omega)$}&\multicolumn{2}{c}{Order}
&\multicolumn{2}{c}{Tolerance}
&\multicolumn{2}{c}{CG}
\\ \hline
1&&2.4&&9&&      4&334e-01&    &&        0&1895&  6\\
2&&2.1&&25&&     3&028e-01&  0.517&&  0&0359&  13\\
3&&1.9&&81&&     1&076e-01&  1.493&&  0&0068&  236\\
4&&1.7&&289&&    2&897e-02&  1.893&&  0&0013&  9896\\
\hline
\end{tabular}
\end{center}
\end{table}
}
\section{Conclusions}
We proved convergence of one-level and multilevel unsymmetric collocation for second order
elliptic boundary value problems. The main theoretical results are Theorem \ref{onelevel} and Theorem \ref{multilevel}.
These theoretical results are obtained with the help of the regularity theory of the solution (Lemmas \ref{trace}-\ref{xianyan}),
the norm equivalence between Sobolev space
$W^{\sigma,2}(\mathbb{R}^{d})$ and the reproducing kernel Hilbert space generated by $\Phi_{\delta}$ (Lemma \ref{mdj}),
the inverse inequality which is satisfied by the functions from kernel-based trial space $W_{\delta}$ (Lemma \ref{invv}),
the sampling inequality (Lemma \ref{wl3}).
Although we obtain $L^{2}$ error of multilevel unsymmetric collocation, it is currently
not yet optimal and deserve refinement. The condition (\ref{restrit}) is too restrictive,
which is not conducive to numerical calculation. In fact, it is caused by the
counteracting term $h_{j}^{-(2+d/2)}$ in error transfer relations (\ref{yichuan1}). It
becomes the price we have to pay for the proof of convergence in multilevel case.
The future proof will eliminate the counteracting term by avoiding the use of (\ref{inv}).

Many papers have performed a large number of numerical experiments to verify the effectiveness
of one-level unsymmetric collocation. Fasshauer's numerical experiments show that
multilevel unsymmetric collocation seems to have
convergence in the absence of the condition (\ref{restrit}) \citep[see][]{f07}, when the experiments
were carrying out under the condition of testing side has the same data sites with the trial side.
But we still need to consider the nonsquare system when implementing unsymmetric collocation because of \citet{hon}.

Many possibilities for enhancement and extension need to be further studied in the future:

(1) The convergence results of multilevel unsymmetric collocation still need improvement by proving 
better Lemma \ref{51}. We predict that the ideal result may be: 
$$\delta_{j}\approx\left(\frac{h_{j}}{\mu}\right)^{1-\frac{2+d/2}{\sigma}},\quad h_{j}=\mu h_{j-1}.$$ 

(2) Observe and prove the condition number of $A^{T}A$ produced by nonsquare collocation matrix $A$. We 
are designing several preconditioners and iterative solvers for solving nonsquare 
systems obtained by unsymmetric  discretization. \citep[see][]{liu22}.

(3) Perform large-scale numerical experiments and compare the observed convergence order with the
theoretical ones of the paper.

\section*{Declarations}
\textbf{Funding}:
The research of the first author was supported
by Natural Science Foundations of China (No.12061057). 
The research of the second author was supported by
the Fourth Batch of the Ningxia Youth Talents Supporting Program (No.TJGC2019012).\\
\textbf{Conflicts of Interests}: The authors declare no conflicts of interest.

\clearpage
\end{document}